\documentclass[12pt]{amsart}
\makeindex
\usepackage[latin1]{inputenc} \usepackage[T1]{fontenc} \usepackage{lmodern}
\usepackage{epsfig,graphicx,color,amsmath,a4wide,amssymb,amsfonts,amsbsy,xy,latexsym,epsf,pstricks,verbatim,times,hyperref}

\usepackage{enumerate}
\usepackage{pgf,color}
\usepackage{pgfarrows}

\xyoption{all}

\definecolor{Grey}{rgb}{.5,.5,.5}

\definecolor{Blue}{rgb}{.0,.0,0.9}
\definecolor{LightBlue1}{rgb}{.2,.4,0.9}
\definecolor{LightBlue2}{rgb}{.3,.5,0.9}
\definecolor{LightBlue3}{rgb}{.4,.6,0.9}
\definecolor{LightBlue4}{rgb}{.5,.7,.9}
\definecolor{LightBlue5}{rgb}{.6,.8,.9}
\definecolor{LightBlue6}{rgb}{.7,.9,.9}

\definecolor{Red}{rgb}{.9,.0,.0}
\definecolor{LightRed1}{rgb}{0.9,.2,.4}
\definecolor{LightRed2}{rgb}{0.9,.3,.5}
\definecolor{LightRed3}{rgb}{0.9,.4,.6}
\definecolor{LightRed4}{rgb}{.9,.5,.7}
\definecolor{LightRed5}{rgb}{.9,.6,.8}
\definecolor{LightRed6}{rgb}{.9,.7,.9}

\xyoption{all}
\usepackage[english]{babel}

\newcounter{noalgo}[section]
\newdimen\indentalgo
\newdimen\indentalgodec\indentalgo=0.0mm\indentalgodec=10mm

\newcommand{\If}{\advance\indentalgo by \indentalgodec {\bf if }}

\newcommand{\For}{\global\advance\indentalgo by \indentalgodec {\bf for }}

\newcommand{\Endindent}{\global\advance\indentalgo by -\indentalgodec}

\newdimen\decalage \decalage=0.5cm
\newcounter{algo} 

\def\<<{\leavevmode
  \raise0.28ex\hbox{$\scriptscriptstyle\langle\!\langle$}\nobreak
  \hskip -.6pt plus.3pt minus.2pt\,}
\def\>>{\,\nobreak\hskip -.6pt plus.3pt minus.2pt
  \raise0.28ex\hbox{$\scriptscriptstyle\rangle\!\rangle$}}

\def\<<{\leavevmode
  \raise0.28ex\hbox{$\scriptscriptstyle\langle\!\langle$}\nobreak
  \hskip -.6pt plus.3pt minus.2pt\,}
\def\>>{\,\nobreak\hskip -.6pt plus.3pt minus.2pt
  \raise0.28ex\hbox{$\scriptscriptstyle\rangle\!\rangle$}}

\newtheorem{definition}{D{e}finition}
\newtheorem{prop}{Proposition}

\newtheorem{theorem}{Theorem}

\newtheorem{lemma}{Lemma}

\def\bK{{\mathbf L}}

\def\bF{{\mathbf F}}

\def\bK{{\mathbf K}}

\def\bK{{\mathbf K}}
\def\bL{{\mathbf L}}

\providecommand{\myproofname}{Proof}

\begin{document}

\author{Tony Ezome and Mohamadou Sall}
\title{On finite field arithmetic in characteristic $2$}

\thanks{Research supported by Simons
 Foundation, Inria International Lab LIRIMA, and ICTP}

\maketitle

\begin{abstract}
We are interested in extending normal bases 
of $\mathbf{F}_{\!2^n}/\mathbf{F}_{\!2}$ to bases
of $\mathbf{F}_{\!2^{nd}}/\mathbf{F}_{\!2}$
which allow fast arithmetic in $\mathbf{F}_{\!2^{nd}}$.
This question has been recently studied
by Thomson and Weir in case $d$ is equal to $2$.
We construct efficient extended bases in case $d$ is
equal to $3$ and $4$. We also give conditions under which
Thomson-Weir construction can be combined with
ours.
\end{abstract}


.

\vspace{.75cm}

\section{Introduction}

Representing elements of a finite
field extension $\mathbf{F}_{\!\!q^m}/\mathbf{F}_{\!\!q}$ by using 
normal bases is adequate when doing arithmetic
in $\mathbf{F}_{\!\!q^m}$.
The main computational advantage of these bases is that they allow fast 
exponentiation by $q$,
this corresponds simply to a cyclic shift of coordinates.
When computing arbitrary products in $\mathbf{F}_{\!\!q^m}$,
Gao, von zur Gathen, Panario and Shoup \cite{Gao-Gathen-Panario-Shoup}
showed that fast multiplication methods such as FFT can be adapted 
to normal bases of $\mathbf{F}_{\!\!q^m}/\mathbf{F}_{\!\!q}$ constructed
from Gauss periods over $\mathbf{F}_{\!\!q}$.
On the other hand, Couveignes and Lercier \cite{Couveignes-Lercier2} 
constructed an FFT-like multiplication algorithm with normal bases
of $\mathbf{F}_{\!\!q^m}/\mathbf{F}_{\!\!q}$ obtained from elliptic curves over $\mathbf{F}_{\!\!q}$.
But the existence of these efficient normal bases puts constraints
on the sizes of $m$ and $q$.
If there is no efficient normal bases
of $\mathbf{F}_{\!\!q^m}/\mathbf{F}_{\!\!q}$ for some $m$ and $q$,
one may hope that $m$ has a proper divisor $n$ 
such that $\mathbf{F}_{\!\!q^n}/\mathbf{F}_{\!\!q}$
admits an efficient normal basis $\mathcal{N} =(\alpha,\alpha^q,
\ldots,\alpha^{q^{n-1}})$. 
Set $m=nd$. Then any basis $B=(\beta_j)_{0\le j\le d-1}$ 
of $\mathbf{F}_{\!\!q^{m}}/\mathbf{F}_{\!\!q^n}$ obviously induces a basis
$\Theta=(\alpha^{q^i}\beta_j)_{i,j}$ of $\mathbf{F}_{\!\!q^{m}}/\mathbf{F}_{\!\!q}$.
This is not a normal basis,
since the $q$-Frobenius automorphism does not act on $B$.
We call such a basis as $\Theta$
an {\it extension of $\mathcal{N}$ with degree $d$}.
In this paper, we construct bases of $\mathbf{F}_{\!\!q^{m}}/\mathbf{F}_{\!\!q}$
by extending normal bases of $\mathbf{F}_{\!\!q^{n}}/\mathbf{F}_{\!\!q}$
and we show that arithmetic operations in $\mathbf{F}_{\!\!q^{m}}$
may be efficiently computed (at least in some cases) by using these extended bases.
Let us recall one of the basics of complexity theory in our context.
Assume that $\Gamma$ is a straight-line program which computes
the coordinates of the product
$x \times y$ in an arbitrary basis $\mathcal{B}$ of $\mathbf{F}_{\!\!q^{m}}/\mathbf{F}_{\!\!q}$
from the ones of $x$ and $y$ by using additions, subtractions,
multiplications of a register by a constant, and additions, subtractions,
multiplications between two registers. Then the  {\it  complexity } of 
$\Gamma$ is the total number of such operations.
The complexity of $\mathcal{B}$ is defined to be 
the minimal possible complexity of a straight-line 
program computing the coordinates of $x \times y$ from the ones of $x$ and $y$.
In addition, we introduce the following terminology.

\begin{definition}
Let $\mathcal{N} =(\alpha^{q^{i}})_{0\le i\le n-1}$ be a normal basis of 
$\mathbf{F}_{\!\!q^{n}}/\mathbf{F}_{\!\!q}$ (this means that $\mathcal{N}$
 is a basis of $\mathbf{F}_{\!\!q^{n}}/\mathbf{F}_{\!\!q}$
generated by the normal element $\alpha$).
\begin{enumerate}[1.]
\item The multiplication table
of $\mathcal{N}$ is defined to be
the matrix \ $T=(t_{i,j})_{0\le i,j\le n-1}$ given by
\begin{equation}\label{eq:7}
\alpha\alpha^{q^i}=\sum_{j=0}^{n-1}t_{i,j}\alpha^{q^j}, \ i=0,1\ldots, n-1.
\end{equation}
\item The weight of $\mathcal{N}$, denoted by $w(\mathcal{N})$, is defined to be
the total number of non-zero entries $t_{i,j}$. 

\item The density of
$\mathcal{N}$, denoted by $d(\mathcal{N})$, is equal to $n \times
w(\mathcal{N})$.

\item For $i,j,k,l\in \{0,\ldots, n-1 \}$,
the products $t_{i,j}t_{k,l}$ are called the cross-products
of the multiplication table of $\mathcal{N}$.

\item Let $\mathcal{B}=(b_i)_{0\le i\le m-1}$
be an arbitrary basis of $\mathbf{F}_{\!\!q^{m}}/\mathbf{F}_{\!\!q}$.
For $ i,j\in \{0,1\ldots, m-1\}$, set
\begin{equation}\label{eq:4}
b_i b_j=\sum_{k=0}^{m-1}t_{i,j}^kb_k.
\end{equation}\begin{enumerate}[(i)]
\item The {\it multiplication tables}
of  $\mathcal{B}$
are defined to be the matrices $(T_0,T_1,\ldots,T_{m-1})$, where
\begin{equation}\label{eq:6}
T_k=(t_{i,j}^k)_{0\le i,j\le m-1}
\end{equation} is defined from equation $(\ref{eq:4})$.
 
\item  The {\it density} of $\mathcal{B}$, denoted by $d(\mathcal{B})$,
is defined to be the total number of non-zero entries $t_{i,j}^k$ for 
$i,j,k\in \{0,\ldots, m-1\}$.
\end{enumerate}
\end{enumerate}
\end{definition}

Addition and substraction of
two elements
$$X=\sum_{0\le k\le m-1}x_kb_k \ \text{ and } \ Y=\sum_{0\le k\le m-1}y_kb_k
\ \text{ in } \ \mathbf{F}_{\!\!q^{m}}$$
expressed in an arbitrary basis $\mathcal{B}=(b_k)_{0\le k\le m-1}$ of
$\mathbf{F}_{\!\!q^{m}}/\mathbf{F}_{\!\!q}$
are performed componentwise and easy to implement.
But multiplication may be more difficult. Set $Z=X\times Y$, and denote by
$\sum_{0\le k\le m-1} z_kb_k$ the decomposition of $Z$ in 
$\mathcal{B}$. The coefficients $z_k$ are obtained from
the multiplication tables $(T_k)_{0\le k\le m-1}$ of $\mathcal{B}$ as follows:
\begin{equation}\label{eq:17}
z_k=XT_k \! \! \ ^tY.
\end{equation}
So the number of operations required to implement multiplication
in $\mathbf{F}_{\!\!q^{m}}$ from the multiplication tables of $\mathcal{B}$
depends on the density of $\mathcal{B}$.
This means that normal bases having low weight have good complexity.
On the other hand, there are quasi-linear time algorithms (for instance the 
one described in \cite{Couveignes-Lercier2})
which output the
coordinates of $X\times Y$ in a normal basis  $\mathcal{N}$
from the ones of
$X$ and $Y$ without using the multiplication table of $\mathcal{N}$.
But if the known normal bases of $\mathbf{F}_{\!\!q^{m}}/\mathbf{F}_{\!\!q}$
have bad complexity, one may turn to
extensions of normal bases of intermediate fields. This means that we first look for
suitable subfields $\mathbf{K}$ of $\mathbf{F}_{\!\!q^{m}}$
containing $\mathbf{F}_{\!\!q}$ such that there
exists an efficient normal basis $\mathcal{N}$ of $\mathbf{K}/\mathbf{F}_{\!\!q}$,
and then we extend $\mathcal{N}$ to a basis of $\mathbf{F}_{\!\!q^{m}}/\mathbf{F}_{\!\!q}$.
In \cite{Thomson-Weir} the authors constructed
extended bases in characteristic $2$ by using Artin-Schreier theory.
So they focused on the case when the degree is equal to $2$.
In the present paper we construct
extended bases whose degree is
equal to $3$ and $4$ 
by using  Kummer theory and Artin-Schreier-Witt theory.
We also give conditions under which
Thomson-Weir construction can be combined with ours.
When the original normal basis $\mathcal{N}$ has subquadratic weight
and subquadratic complexity,
we show that all the resulting extended bases have subquadratic complexity.

\subsection*{Plan}

In Section 2 we present quadratic
Artin-Schreier extended bases and degree $4$ Artin-Schreier-Witt extended bases.
In Section 3 we describe degree $3$ Kummer extended bases.
Section 4 is devoted to extended bases in the context of
towers of field extensions obtained from Artin-Schreier and Kummer theories.

\subsection*{Notation:} 
Throughout this paper $\bK$ denotes a field with characteristic $p>0$,
and $\overline{\bK}$ is an algebraic closure of $\bK$.

\section{Extended bases whose degree is a power of $2$}

In this section we recall general results concerning cyclic
extensions of $\bK$ whose degree is a $p$-power,
and we specify the case when $\bK$ is a finite field with characteristic $2$.

\subsection{Artin-Schreier extended bases in characteristic $2$}\label{sect:1}

It is proved in [\cite{Lang}, Chapter VI, Theorem 6.4] that any degree $p$ cyclic extension
of $\bK$ is generated by a root of a polynomial of the form
$$X^p-X-\alpha,$$
where $\alpha\in \bK$ lies outside of the set $\{x^p-x \ \vert  x\in \bK \}$.
Irreducible polynomials of this type are useful both
for constructing efficient normal bases and for extending them.
For instance in [\cite{Ezome-Sall1}, Theorem 1] the authors constructed a normal basis
of $\bK[X]/(X^p-X-\alpha)$ over $\bK$ with low weight and quasi-linear complexity.
On the other hand, \textit{degree $p$ Artin-Schreier extended bases} defined below
are constructed from irreducible polynomials of the form $X^p-X-\alpha$.

\begin{definition}
Let $p$ be a prime number and $q$ a power of $p$.
Let $\mathcal{N} =(\alpha^{q^{i}})_{0\le i\le n-1}$ be a normal basis of 
$\mathbf{F}_{\!\!q^{n}}/\mathbf{F}_{\!\!q}$. Denote by $\overline{\bF}_{\!\!q}$
an algebraic closure of $\bF_{\!\!q}$ containing $\mathbf{F}_{\!\!q^{n}}$.
A degree $p$ Artin-Schreier extension
of $\mathcal{N}$ (also Artin-Schreier extended basis)
is a basis $\mathcal{A}$  of  $\mathbf{F}_{\!\!q^{np}}/\mathbf{F}_{\!\!q}$ 
for which there exists $\beta$ in $\overline{\bF}_{\!\!q}$ outside of $\mathbf{F}_{\!\!q^n}$
such that $\beta^p-\beta=\alpha$ and $\mathcal{A} = (\alpha^{q^i}\beta^j)_{i,j} $.
\end{definition}

It is shown that
any normal basis $\mathcal{N}=(\alpha, \alpha^{2}, \ldots,\alpha^{2^{n-1}})$
of $\bF_{\!2^{n}}/\bF_{\!2}$ admits
an Artin-Schreier extension. Indeed,
assume that the polynomial $f(X)=X^2+X+\alpha$ is reducible over $\bF_{\!2^n}$.
Then the additive form of
Hilbert's Theorem 90 ensures that
$\mathrm{Tr}_{\bF_{\!2^n}/\bF_{\!2}}(\alpha)=0$.
But this is impossible since $\alpha$ is a normal element of $\bF_{\!2^n}/\bF_{\!2}$.
Hence any $\beta$  in $\bF_{\!2^{2n}}$ satisfying
$$\beta^2+\beta=\alpha$$
defines a quadratic Artin-Schreier extension
$\mathcal{A}=\mathcal{N}\cup \beta\mathcal{N}$ of $\mathcal{N}$.
The following statement describes squaring in $\bF_{\!2^{2n}}$, it also
gives the complexity and density of  $\mathcal{A}$.

\begin{prop}\label{prop:1}
Let $\mathcal{N}=(\alpha, \alpha^{2}, \ldots,\alpha^{2^{n-1}})$
be a normal basis of $\bF_{\!2^{n}}/\bF_{\!2}$
and $\beta$ an element  in $\bF_{\!2^{2n}}$ such that
 $\mathcal{A}=\mathcal{N}\cup \beta\mathcal{N}$ is
a degree $2$ Artin-Schreier extension of $\mathcal{N}$.
\begin{enumerate}[1.]
\item Squaring in $\bF_{\!2^{2n}}$ is given by
$$(C+\beta D)^2 = (C_{>}+E)+\beta D_{>},$$
where $C_{>}$ and  $D_{>}$ stand for right-cyclic shifts of the coordinate vectors
of \ $C$ and $D$, and
 $$E = \  ^t\!D_{>} \times T$$
is a vector-matrix multiplication between the transpose of $D_{>} $
and the multiplication table of $\mathcal{N}$.
\item The complexity of $\mathcal{A}$ consists in at most:
\begin{enumerate}
\item 3 multiplications and 4 additions between elements
lying in $\bF_{\!2^{n}}$;
\item 1 vector-matrix multiplication between a vector of $\bF_{\!2^n}/\bF_{\!2}$
and the multiplication table of $\mathcal{N}$.
\end{enumerate}

\item If $\mathcal{N}$ has subquadratic complexity
and subquadratic weight, then
$\mathcal{A}$ has also subquadratic complexity.

\item Let $w(\mathcal{N})$ be the weight of $\mathcal{N}$.
For $(i,\delta)\in \{0,\ldots,n-1 \} \times \{0,1 \}$,
set $\mathfrak{a}_{i+\delta n}=\alpha^{2^i}\beta^\delta$ so that
$\mathcal{A}=(\mathfrak{a}_{k})_{0\le k\le 2n-1}$.
\begin{enumerate}
\item For $0\le k\le n-1$, the number of non-zero entries in the
 $k$-th multiplication table of $\ \mathcal{A}$, is equal to
$$
w(\mathcal{N})+ \sum_{0\le i,j\le n-1}
\varphi(\sum_{r=0}^{n-1}t_{j-i,r-i}t_{r,k}),$$
where subscripts are taken modulo $n$,
$(t_{i,j})_{0\le i,j\le n-1}$ is the multiplication table
of $\mathcal{N}$, and 
$\varphi$ is the unique ring homomorphism
from $\bF_{\!2}$ into  $\mathbb{Z}$.
\item The $k$-th multiplication table of \ $\mathcal{A}$, for $n\le k\le 2n-1$,
 has $3w(\mathcal{N})$ non-zero entries.
\end{enumerate}
\end{enumerate}
The density of $\mathcal{A}$ is given by
$$d(\mathcal{A})=4d(\mathcal{N})
+ \sum_{0\le k\le n-1}\sum_{0\le i,j\le n-1}\varphi(\sum_{0\le r\le n-1} t_{j-i,r-i}t_{r,k}).$$
\end{prop}

\begin{proof}
\begin{enumerate}[1.]
\item This is [\cite{Thomson-Weir}, Proposition 3.7].
 Let $C=\sum_{i=0}^{n-1}c_i\alpha^{2^i}$ and 
$D=\sum_{i=0}^{n-1}d_i\alpha^{2^i}$
 be the linear combinations of $C$ and $D$
with respect to $\mathcal{N}$. We have
$$
\begin{array}{rl}
(C+\beta D)^2 = & \sum_{i=0}^{n-1}c_{i-1}\alpha^{2^i} + \beta^2\sum_{i=0}^{n-1}d_{i-1}\alpha^{2^i}\\
= & \sum_{i=0}^{n-1}c_{i-1}\alpha^{2^i} + (\beta+\alpha)\sum_{i=0}^{n-1}d_{i-1}\alpha^{2^i}\\
= & \Big(\sum_{i=0}^{n-1}c_{i-1}\alpha^{2^i} +\sum_{i=0}^{n-1}d_{i-1}\alpha\alpha^{2^i} \Big)
+\beta \sum_{i=0}^{n-1}d_{i-1}\alpha^{2^i}.
\end{array}
$$
So
$$(C+\beta D)^2 = \Big(\sum_{i=0}^{n-1}c_{i-1}\alpha^{2^i}+\sum_{i=0}^{n-1}d_{i-1}\sum_{k=0}^{n-1}
t_{i,k}\alpha^{2^k}\Big) + \beta \sum_{i=0}^{n-1}d_{i-1}\alpha^{2^i},$$
where subscripts are taken modulo $n$ and
 $(t_{i,k})_{i,k}$ stands for the multiplication 
table of $\mathcal{N}$.
The term $\sum_{i=0}^{n-1}d_{i-1}\sum_{k=0}^{n-1}
t_{i,k}\alpha^{2^k}$ corresponds to
a vector-matrix multiplication between the transpose of the right-cyclic
shift of the coordinate vector of $D$ and
the multiplication table of $\mathcal{N}$.
Assume that $\mathcal{N}$ has subquadratic weight in $n$. This means
that its multiplication table
is a sparse matrix with $o(n^2)$ non-zero entries. So the computation
of the above vector-matrix
multiplication needs $o(n^2)$ operations in $\bF_{\!2}$.
Since a cyclic shift of coordinates of a vector in $\bF_{\!2^{n}}$ over $\bF_{\!2}$
runs in time $O(n)$, we
conclude that squaring in $\bF_{\!2^{2n}}$ has subquadratic running time.

\item Let $C=C_0+\beta C_1$ and $D=D_0+\beta D_1$ be two elements
of $\bF_{\!2^{2n}}$ expressed in
$\mathcal{A}$. A Karatsuba-like multiplication algorithm gives
\begin{equation}\label{eq:2}
\begin{array}{rl}
C\times D = & \beta^2 C_1D_1+\beta \big( (C_1+C_0)(D_1+D_0)+C_1D_1 +C_0D_0 \big)+C_0D_0\\
= & (\beta^2+\beta)C_1D_1+\beta \big( (C_1+C_0)(D_1+D_0)+C_0D_0\big)+C_0D_0.
\end{array}
\end{equation}
Since $\beta^2+\beta=\alpha$, we have
\begin{equation}\label{eq:9}
C\times D= C_0D_0+\alpha C_1D_1+\beta \Big( (C_1+C_0)(D_1+D_0)+C_0D_0\Big).
\end{equation}
So the product $C \times D$ consists in 3 multiplications and 4 additions
between elements in $\bF_{\!2^{n}}$, and a vector-matrix 
multiplication which corresponds to the term $\alpha C_1 D_1$ in equation $(\ref{eq:9})$.

\item We described
the computation of the vector-matrix multiplication $\alpha C_1 D_1$
 in the first item of the present proof.
Thus if $\mathcal{N}$ has subquadratic weight and subquadratic complexity,
then a multiplication in $\bF_{\!2^{n}}$ needs $o(n^2)$ operations in $\bF_{\!2}$,
as well as a vector-matrix multiplication.
Since sum of two vectors in $\bF_{\!2^{n}}$ over $\bF_{\!2}$
can be computed in time $O(n)$, we
conclude that the product $C \times D$ needs $o(n^2)$ operations in $\bF_{\!2}$.

\item For $(i,\delta)\in \{0,\ldots,n-1 \} \times \{0,1 \}$,
we set $\mathfrak{a}_{i+\delta n}=\alpha^{2^i}\beta^\delta$ so that
$\mathcal{A}=(\mathfrak{a}_{k})_{0\le k\le 2n-1}$.
Hence, the multiplication tables $T_k$ of $\mathcal{A}$ are given by the block matrix
$$\left( \begin{array}{c|c}
 & \\ 
(\alpha^{2^i}\alpha^{2^j})_{0\le i,j\le n-1} & (\beta\alpha^{2^i}\alpha^{2^j})_{0\le i,j\le n-1}\\ 
 & \\ \hline
& \\ 
(\beta \alpha^{2^i}\alpha^{2^j})_{0\le i,j\le n-1} & ((\alpha+\beta)\alpha^{2^i}\alpha^{2^j})_{0\le i,j\le n-1}\\ 
 & 
\end{array} \right)
$$
\begin{enumerate}
\item For $0\le k\le n-1$, the components of the matrix
$T_k$ come from the blocks $(\alpha^{2^i}\alpha^{2^j})_{0\le i,j\le n-1}$ and 
$((\alpha+\beta)\alpha^{2^i}\alpha^{2^j})_{0\le i,j\le n-1}$. 
On the other hand, the multiplication table $T=(t_{r,s})_{0\le r,s\le n-1}$ of $\mathcal{N}$
is given by
$$
\alpha\alpha^{2^r}=\sum_{s=0}^{n-1}t_{r,s}\alpha^{2^s}, \ 0\le r\le n-1.
$$
Since
$$
\alpha^{2^i}\alpha^{2^j}=\sum_{r=0}^{n-1}t_{i,j}^r\alpha^{2^r},
$$
it follows that $t_{i,j}^r=t_{j-i,r-i}$. So
\begin{equation}\label{eq:8}
\alpha \alpha^{2^i}\alpha^{2^j}=\alpha\sum_{r=0}^{n-1}t_{j-i,r-i}\alpha^{2^r}=
\sum_{r=0}^{n-1}t_{j-i,r-i}\sum_{k=0}^{n-1}t_{r,k}\alpha^{2^k},
\end{equation}
where subscripts are taken modulo $n$.
The number of non-zero entries  in the matrix $T_k$ coming from
$\alpha \alpha^{2^i}\alpha^{2^j}$ is given by the coefficient
of $\alpha^{2^k}$ in the linear combination $(\ref{eq:8})$.
This number is equal to 
$$\varphi(\sum_{r=0}^{n-1} t_{j-i,r-i}t_{r,k}),$$
where $\varphi$ is the unique ring homomorphism
from $\bF_{\!2}$ into  $\mathbb{Z}$.
Hence the total number of non-zero entries in $T_k$ is equal to
$$w(\mathcal{N})+ \sum_{0\le i,j\le n-1}
\varphi(\sum_{r=0}^{n-1}t_{j-i,l-i}t_{r,k}).$$

\item For $n\le k\le 2n-1$, the components of
$T_k$ come from the blocks $(\beta\alpha^{2^i}\alpha^{2^j})_{0\le i,j\le n-1}$
and $((\beta+\alpha) \alpha^{2^i}\alpha^{2^j})_{0\le i,j\le n-1}$.
So the number of non-zero entries in $T_k$ is equal to
$3w(\mathcal{N}).$
\end{enumerate}
\end{enumerate}
\end{proof}

\subsection{Background on Witt vectors}\label{sect:4}

Witt described in \cite{Witt} cyclic field extensions 
whose degree is a power of the characteristic of the base field. 
Theorem \ref{thm:2} below is  one of the main results of \cite{Witt}.
Let $A$ be a commutative ring with unit element, and $S$ a (possibly infinite)
subset of the natural numbers $\mathbb{N}$.
The structure of commutative ring with unit on the cartesian product $A^{S}$
is easily verified, addition and multiplication are performed componentwise.
There may be other ring structures on $A^S$, for instance the one from the theory
of Witt vectors (see \cite{Serre-Corps-Locaux} or \cite{Lara-PhD}).
We let $p$ be a prime number and $\phi$ the unique ring-homomorphism
from the integers into $A$.
For any $n$ in $\mathbb{Z}$,  we write again $n$ instead of $\phi(n)$.
We start with the assumption that $p$ is invertible in $A$.
Then the set of Witt vectors with components in $A$ denoted by
$W(A)$ is the set of sequences $\mathbf{x}=
(x_k)_{k\in \mathbb{N}}$ of elements of
$A$ which admit {\it  sequences of
ghost components} $(x^{(k)})_{k\in \mathbb{N}}$ defined by
\begin{equation}\label{eq:13}
x^{(k)}:=x_0^{p^k}+ px_1^{p^{k-1}}+\ldots+p^kx_k,
\end{equation}
On the other hand, it is easily seen that
\begin{equation}\label{eq:16}
x_0=x^{(0)}, \ \  x_1=\frac{1}{p}\Big(x^{(1)}-x_0^{p}\Big) \  \ \text{ and } \ \
x_k=\frac{1}{p^k}\Big( x^{(k)}-\sum_{0 \le d \le k-1} p^dx_d^{p^{k-d}}\Big)
 \ \text{ for any } \ k\ge 1.
\end{equation}
So components of a Witt vector are recursively computed from
its ghost components and vice versa.
We deduce that the map
$$
\xymatrix{
\varphi : W(A) & \ar@{->}[r] & 
A^{\mathbb{N}}\\
\mathbf{x}=(x_k)_{k\in \mathbb{N}} & \ar@{|-{>}}[r] & (x^{(k)})_{k\in \mathbb{N}}
}$$
is a bijection. From the structure of product ring 
on $A^{\mathbb{N}}$, we obtain
a structure of commutative ring with unit on $W(A)$
whose composition laws are given by
$$
\mathbf{x}+\mathbf{y}=\varphi^{-1}((x^{(k)})_k+(y^{(k)})_k) \quad \text{ and } \quad
\mathbf{x} \times \mathbf{y}=\varphi^{-1}((x^{(k)})_k\times(y^{(k)})_k).
$$
Actually the components of the sum and product
of two Witt vectors $\mathbf{x}$ and $\mathbf{y}$
may be computed from polynomial equations involving the components of
$\mathbf{x}$ and $\mathbf{y}$,  for a more detailed exposition of this fact
see \cite{Witt}, \cite{Serre-Corps-Locaux} or \cite{Lara-PhD}.
It is shown that there exists a unique sequence 
$S_0, S_1, \ldots, S_n, \ldots$ of polynomials in 
$\mathbb{Z}[X_0,X_1, \ldots, X_n, \ldots; Y_0,Y_1, \ldots, Y_n, \ldots, ]$
 (resp. a unique sequence
$P_0, P_1,\ldots, P_n,\ldots)$
so that for $\mathbf{x}$ and $\mathbf{y}$  in $W(A)$ we have
\begin{equation}\label{eq:15}
(\mathbf{x}+ \mathbf{y})_k=S_k(\mathbf{x},\mathbf{y})
\ \ \text{ and } \ \ (\mathbf{x}\times \mathbf{y})_k=P_k(\mathbf{x},\mathbf{y}).
\end{equation}
In both cases, the first two polynomials $S_0$ and $S_1 ($ resp.  $P_0$ and $P_1)$
can be easily computed:
\begin{equation}\label{eq:14}
\begin{array}{l}
S_0(\mathbf{x}, \mathbf{y})=x_0+y_0, \quad S_1(\mathbf{x}, \mathbf{y})
=x_1+y_1+\frac{1}{p}\sum_{k=1}^{p-1}
\Big( \begin{array}{l}
p\\
k \end{array} \Big) x_0^{k}y_0^{p-k},\\
P_0(\mathbf{x}, \mathbf{y})=x_0y_0, \quad P_1(\mathbf{x}, \mathbf{y})
= x_1y_0^p+y_1x_0^p+px_1y_1.
\end{array}
\end{equation}
In case $A$ is an arbitrary commutative ring with unit (even if $p$ is not invertible),
it is shown that $W(A)$ is also a commutative ring with unit, the laws being
defined from the polynomials $S_k$ and  $P_k$ in equation $(\ref{eq:15})$
(see [\cite{Serre-Corps-Locaux}, Chapter II,  \S 6]
 or [\cite{Lara-PhD}, Section 1.1]).
 Since the polynomials $S_k$ and  $P_k$ only involve variables $X_k$ and $Y_k$ whose
 index are $\le k$, we deduce that for any positive integer
 $r$ the set $W_{\!r}(A)$ of truncated Witt vectors
 $(x_0,x_1,\ldots, x_{r-1})$ with length $r$ and components
 in $A$ form a commutative ring with unit.
In case $A$ is a field with characteristic
$p$, the group-homomorphism 
$$\begin{array}{llll}
\wp: & A & \longrightarrow & A\\
 & x & \mapsto & x^p-x.
\end{array}
$$
induces a group-homomorphism from $W_{\!r}(A)$
into itself that we call $\wp$ also.
The following theorem generalizes Artin-Schreier Theorem.

\begin{theorem}\label{thm:2}
Let $K$ be a field with characteristic $p>0$, and $r\ge 1$ an integer.
\begin{enumerate}[1.]
\item Let $x=(x_0,x_1,\ldots,x_{r-1})$ be a truncated Witt vector with components in $K$.
\begin{enumerate}
\item The equation 
\begin{equation}\label{eq:5}
\wp(\xi)=x
\end{equation}

either has no root 
in $W_{\!r}(K)$, or it has a root in $W_{\!r}(K)$. In the later case, all its 
$p^r$ roots lie in $W_{\!r}(K)$.

\item If equation $(\ref{eq:5})$ has no root in $W_r(K)$,
then $K(\wp^{-1}(x))$ is a cyclic extension of $K$ with degree
dividing $p^r$.
The degree $[K(\wp^{-1}(x)) : K]$ is equal to $p^r$ if and only if
$x_0\notin \wp(K)$.
\end{enumerate}

\item If $L/K$ is a cyclic extension with degree $p^r$,
then there exists $x$ in $W_{\!r}(K)$ such that $L=K(\wp^{-1}(x))$ 
and $x_0 \notin \wp(K)$.
\end{enumerate}
\end{theorem}

\begin{proof}
See \cite{Witt}, [\cite{Lang}, Page 331] or [\cite{Lara-PhD}, Section 2.1.1].
\end{proof}

\subsection{Artin-Schreier-Witt extended bases in characteristic $2$}\label{sect:4}

We first introduce the following terminology.

\begin{definition}
Let $p$ be a prime number and $q$ a power of $p$.
Let $\mathcal{N} =(\alpha^{q^{i}})_{0\le i\le n-1}$ be a normal basis of 
$\mathbf{F}_{\!\!q^{n}}/\mathbf{F}_{\!\!q}$. Denote by $\overline{\bF}_{\!\!q}$
an algebraic closure of $\bF_{\!\!q}$ containing $\mathbf{F}_{\!\!q^{n}}$.
Let $(\beta_1, \ldots, \beta_{r}) \in W_r(\overline{\mathbf{F}_{\!\!q}})$
be a truncated Witt vector outside of $W_r(\mathbf{F}_{\!\!q^n})$
such that
$$(\beta_1^p, \ldots, \beta_{r}^p)-(\beta_1, \ldots, \beta_{r})
=(\alpha,x_1 \ldots, x_{r-1})$$
where $x_1, \ldots, x_{r-1}$ are arbitrary elements in $\mathbf{F}_{\!\!q^n}$.
Set 
$\mathcal{W}_1=\mathcal{N}\cup\beta_1\mathcal{N}\cup \ldots \cup \beta_1^{p-1}
\mathcal{N}$ and
$$\mathcal{W}_i=\mathcal{W}_{i-1}\cup\beta_i\mathcal{W}_{i-1} \cup
\ldots  \cup \beta_i^{p-1} \mathcal{W}_{i-1}, \ \text{for any }  i\in \{2,\ldots, r \}$$
so that $\mathcal{W}_i$ is a basis of $\mathbf{F}_{\!\!q^{np^{i}}}/\mathbf{F}_{\!\!q^{n}}.$
Such a basis $\mathcal{W}_i$ is called a degree $p^i$ Artin-Schreier-Witt
extension of $\mathcal{N}$ (also Artin-Schreier-Witt extended basis).
\end{definition}

In this section, we focus on the case when the lenght of the truncated Witt vectors and
the characteristic of the base field are equal to $2$.
Given two truncated Witt vectors $\mathbf{x}=(x_0,x_1)$ and $\mathbf{y}
=(y_0,y_1)$ in $W_{\!2}(\overline{\bF}_{\!2^n})$,
we know from equation $(\ref{eq:14})$ that
\begin{equation}\label{eq:11}
\mathbf{x}+\mathbf{y}=(x_0+y_0,x_1+y_1+x_0y_0).
\end{equation}
Theorem \ref{thm:2} tells us that  constructing degree $4$
Artin-Schreier-Witt extensions is related to solving equations
of the form $\wp(\xi)=x$ in $W_{\!2}(\overline{\bF}_{\!2^n})$.
Let $\mathcal{N}=(\alpha, \alpha^{2}, \ldots,\alpha^{2^{n-1}})$ be
a normal basis of $\bF_{\!2^n}/\bF_{\!2}$.
Assume that $(\beta_0, \beta_1)$ is a truncated Witt vector
in $W_{\!2}(\overline{\bF}_{\!2^n})$ such that\\

\noindent $
(14) \qquad \qquad \qquad  \qquad  \qquad  \quad
(\beta_0^2, \beta_1^2)+(\beta_0,\beta_1)=(\alpha,\alpha).
$\\

\noindent Set $(s_0,s_1)=(\beta_0^2, \beta_1^2)+(\beta_0,\beta_1)$. 
From equation $(\ref{eq:11})$, we obtain
$$s_0=\beta_0^2+\beta_0 \ 
\text{ and } s_1=\beta_1^2+\beta_1+\beta_0^3.$$
By setting $(s_0', s_1')=(s_0, s_1)+(\alpha, \alpha)$,
we find
$$s_0'=\beta_0^2+\beta_0+\alpha \ \text{ and } \
s_1'=\beta_1^2+\beta_1+\beta_0^3+\alpha+\alpha\beta_0^2+\alpha\beta_0.$$
Hence equation $(\ref{eq:12})$ yields
$$\beta_0^2=\beta_0+ \alpha \ \text{ and } \
\beta_1^2= \beta_1+\beta_0(1+\alpha)+\alpha^2.$$

Squaring in $\mathbf{F}_{\!\!2^{4n}}$
and the complexity of a degree $4$ Artin-Schreier-Witt extension of
$\mathcal{N}$ are described in the following statement.

\begin{prop}\label{prop:4}
\begin{enumerate}[1.]
 Let $p$ be a prime number and $q$ a $p$-power. 
Let $\mathcal{N}=(\alpha, \alpha^{p}, \ldots,\alpha^{p^{n-1}})$
be a normal basis of $\bF_{\!p^n}/\bF_{\!p}$.
\item $\mathcal{N}$ admits an Artin-Schreier-Witt extension with degree $q$.
\item Assume $p=2$ and $q=4$. 
Denote by $\overline{\mathbf{F}}_{\!\!2^n}$
an algebraic closure of $\mathbf{F}_{\!\!2}$ containing $\mathbf{F}_{\!\!2^n}$.
Let $(\beta_0, \beta_{1})$ be a truncated Witt vector in
 $W_2(\overline{\bF}_2)$ outside of $W_2(\mathbf{F}_{\!\!2^n})$ and
such that
\begin{equation}\label{eq:12}
(\beta_0^2, \beta_1^2)+(\beta_0,\beta_1)=(\alpha,\alpha).
\end{equation}
Denote by $\mathcal{W}=(\mathcal{N}\cup \beta_0
\mathcal{N})\cup \beta_1(\mathcal{N}\cup \beta_0\mathcal{N})$
the corresponding degree $4$ Artin-Schreier-Witt extension  of $\mathcal{N}$. 
\begin{enumerate}[(a)]
\item If $\gamma= A+ \beta_0 B+ \beta_1(C +\beta_0 D)$ is an element
of $\bF_{\!2^{4n}}$ expressed in
$\mathcal{W}$, then
$$\gamma^2=  \Bigg[ A_{>} + \alpha B_{>} + \alpha^2 C_{>}
$$
$$
 + \ \  (\alpha^3+\alpha^2+\alpha) D_{>} +  \beta_0\big( B_{>} +
 (1+\alpha) C_{>}+ D_{>} \big) \Bigg] \ \ + \ \
\beta_1 \Bigg[ C_{>} + \alpha D_{>}+\beta_0 D_{>} \Bigg],$$
where $A_{>}, B_{>}, C_{>}$ and $D_{>}$ stand for right-cyclic
shifts of the coordinate vectors of $A, B, C$ and $D$.
\item The complexity of $\mathcal{W}$ consists in at most:
\begin{itemize}
\item 9 multiplications and 33 additions between elements
lying in $\bF_{\!2^{n}}$;
\item 9 vector-matrix multiplications between a vector of $\bF_{\!2^n}$
and the multiplication table of $\mathcal{N}$.
\end{itemize}
\end{enumerate}
\item If $\mathcal{N}$ has subquadratic complexity
and subquadratic weight, then
$\mathcal{W}$ has also subquadratic complexity.
\end{enumerate}
\end{prop}

\begin{proof}
\begin{enumerate}[1.]
\item One shows that $\alpha$ lies outside of $\{x^p-x \ \vert  x\in \bF_{\!p^n} \}$
by using Hilbert's Theorem 90
as in the beginning of Section \ref{sect:1}.
Let $r\ge 1$ be an integer such that $q=p^r$. Let 
$x=(\alpha,x_1,x_2,\ldots,x_{r-1})$ be a truncated Witt vector with components in
$\bF_{\!p^n}$. By Theorem \ref{thm:2}, we conclude that $\bF_{\!p^n}(\wp^{-1}(x))$
is a degree $q$ cyclic extension of $\bF_{\!p^n}$.
\item \begin{enumerate}[(a)]
\item We have
$$\big(A+ \beta_0 B+ \beta_1(C +\beta_0 D)\big)^2=
  A_{>}+ \beta_0^2 B_{>}+ \beta_1^2(C_{>} +\beta_0^2 D_{>}),$$
where $A_{>}$, $B_{>}$, $C_{>}$ and $D_{>}$ stand for right-cyclic shifts of the 
coordinate vectors of $A$, $B$, $C$ and $D$.
We saw that equation $(\ref{eq:12})$ implies
$$\beta_0^2=\beta_0+ \alpha \ \text{ and } \
\beta_1^2= \beta_1+\beta_0(1+\alpha)+\alpha^2.$$
So
$$
\beta_1^2(C_{>} +\beta_0^2 D_{>})=
\big[ \alpha^2 C_{>} +(\alpha^3+\alpha^2+\alpha) D_{>} + 
\beta_0 \big( (1+\alpha) C_{>}+ D_{>}  \big)\big]$$
$$ + \ \ \beta_1\big[C_{>} + \alpha D_{>}+\beta_0 D_{>}\big].
$$
Hence
$$\begin{array}{llr}
\big(A+ \beta_0 B+ \beta_1(C +\beta_0 D)\big)^2  & = &
  \Bigg[ A_{>} + \alpha B_{>} + \alpha^2 C_{>}
\end{array}$$
$$
 + \ \  (\alpha^3+\alpha^2+\alpha) D_{>} +  \beta_0\big( B_{>} + 
 (1+\alpha) C_{>}+ D_{>} \big) \Bigg] \ \ + \ \
\beta_1 \Bigg[ C_{>} + \alpha D_{>}+\beta_0 D_{>} \Bigg].$$

From the study made in the proof of
Proposition \ref{prop:1}, we dedude that the terms
 $P(\alpha) X$ (for $P(\alpha)$ a non-constant polynomial in $\bF_{\!2}[\alpha]$ with degree
$\le 3$ and $X$ a vector in $\bF_{\!2^n}$) correspond 
to sums of vectors of the form $\alpha^{i} X$ with $1\le i\le 3$. Each such
vector $\alpha^{i} X$ corresponds to
$i$ vector-matrix multiplications between 
vectors in $\bF_{\!2^{n}}$ and the multiplication table of $\mathcal{N}$.

\item  From a Karatsuba-like multiplication method,
the product of two elements 
$$
X_1=(A_1+\beta_0 B_1)+\beta_1 (C_1+\beta_0 D_1)   \ \ \text{ and } \ \
X_2=(A_2+\beta_0 B_2)+\beta_1 (C_2+\beta_0 D_2) \ \text{ in } \ \bF_{\!2^{4n}}
$$
is given by
$$
\begin{array}{lll}
X_1 \times X_2 & = & \beta_1^2(C_1+\beta_0 D_1)(C_2+\beta_0 D_2)\\
& & + \beta_1 \Big[ (A_1+\beta_0 B_1+C_1+\beta_0 D_1) 
(A_2+\beta_0 B_2 + C_2+\beta_0 D_2)\\
 & &+ (A_1+\beta_0 B_1)(A_2+\beta_0 B_2)
+(C_1+\beta_0 D_1)(C_2+\beta_0 D_2)\Big]\\
& & +(A_1+\beta_0 B_1)(A_2+\beta_0 B_2),
 \end{array}
$$
that is
$$
\begin{array}{lll}
X_1 \times X_2 & = &
\beta_1^2
\Bigg[ \beta_0^2D_1D_2 +\beta_0\big((C_1+D_1)(C_2+D_2)+C_1C_2+D_1D_2\big)+ C_1C_2 \Bigg]
 \end{array}$$
$$\begin{array}{ll}
+ & \beta_1 \Bigg[  \beta_0^2 (B_1+D_1)(B_2+D_2)+\beta_0\Big((A_1+B_1+C_1+D_1) (A_2+B_2+C_2+D_2)\\
 + & (A_1+C_1)(A_2+C_2) +(B_1+D_1)(B_2+D_2) \Big) + (A_1+C_1)(A_2+C_2)\\
+ &  \beta_0^2B_1B_2 + \beta_0\Big((A_1+B_1)(A_2+B_2)+ A_1A_2+B_1B_2)\Big)+A_1A_2+\beta_0^2D_1D_2\\
+ & \beta_0\Big((C_1+D_1)(C_2+D_2)+ C_1C_2+D_1D_2\Big)+ C_1C_2 \Bigg]+
  \beta_0^2B_1B_2 \\
+ &  \beta_0\Big((A_1+B_1)(A_2+B_2)+ A_1A_2+B_1B_2\Big)+A_1A_2.
 \end{array}
$$
Since $\beta_0^2=\beta_0+ \alpha$ and 
$\beta_1^2= \beta_1+\beta_0(\alpha+1)+\alpha^2$, we have
$$
\begin{array}{lll}
X_1 \times X_2 & = &
\Bigg[ A_1A_2+\alpha B_1B_2+\alpha^2 C_1C_2  \end{array}$$
$$\begin{array}{ll}
+ & (\alpha^3+\alpha^2+\alpha)D_1D_2 +(\alpha^{2}+\alpha)
\big((C_1+D_1)(C_2+D_2)+C_1C_2+D_1D_2\big)\\
+ &  \beta_0 \Big( A_1A_2 + (\alpha+1)C_1C_2+D_1D_2+(A_1+B_1)(A_2+B_2)\\
+ & (\alpha^2+\alpha+1)\Big( (C_1+D_1)(C_2+D_2)+C_1C_2+D_1D_2\Big)\Bigg]\\
+  &  \beta_1 \Bigg[ A_1A_2+ \alpha B_1B_2 + C_1C_2+ \alpha D_1D_2 + (A_1+C_1)(A_2+C_2)\\
+ &  \alpha (B_1+D_1)(B_2+D_2)+ \beta_0\Big(A_1A_2 + C_1C_2+  (A_1+B_1)(A_2+B_2)
+ (A_1+C_1)(A_2+C_2)\\
+ & (C_1+D_1)(C_2+D_2) + (A_1+B_1+C_1+D_1) (A_2+B_2+C_2+D_2) \Big)
\Bigg].\end{array}$$
For $1\le i \le 3$ and $X$ a vector in $\bF_{\!2^n}$,
 $\alpha^{i} X$ corresponds to $i$ vector-matrix multiplications between 
vectors in $\bF_{\!2^n}$ and the multiplication table of $\mathcal{N}$.
So the computation of $X_1\times X_2$ consists in:
\begin{itemize}
\item 9 multiplications and 33 additions between elements
lying in $\bF_{\!2^{n}}$;
\item 9 vector-matrix multiplications between a vector of $\bF_{\!2^n}$
and the multiplication table of $\mathcal{N}$.
\end{itemize}
\end{enumerate}
\item The same argument as in the proof of
Proposition \ref{prop:1} shows that a normal basis with subquadratic weight and
subquadratic complexity yields Artin-Schreier-Witt extended bases with subquadratic
complexity.
\end{enumerate}
\end{proof}

The density of the degree $4$ Artin-Schreier-Witt extended basis
$$\mathcal{W}=(\mathcal{N}\cup \beta_0
\mathcal{N})\cup \beta_1(\mathcal{N}\cup \beta_0\mathcal{N})$$
described in Proposition \ref{prop:4}  is given by Lemma \ref{lem:3} below.
For $(i, \delta, \lambda)\in \{0,\ldots, n-1\} \times \{0,1 \} \times \{0,1 \}$,
we set $\mathrm{w}_{i+\delta n+\lambda n}=\alpha^{2^i}\beta_0^\delta\beta_1^\lambda$ so that 
$\mathcal{W}=(\mathrm{w}_\ell)_{0\le \ell\le 4n-1}$.

\begin{lemma}\label{lem:3}
With the above notation, 
let $w(\mathcal{N})$ be the weight of the normal basis $\mathcal{N}$.
\begin{enumerate}[1.]
\item For $0\le \ell\le n-1$, the number of non-zero entries in the
 $\ell$-th multiplication table of $\mathcal{W}$, is equal to
$$\begin{array}{rl}
 & w(\mathcal{N})+\sum_{0\le i,j\le n-1} \varphi( \sum_{r=0}^{n-1}t_{j-i,r-i}t_{r,\ell})\\
+ & \sum_{0\le i,j\le n-1} \varphi \Bigg(
 \sum_{r=0}^{n-1}\ \sum_{s=0}^{n-1}t_{j-i,r-i}t_{r, s} t_{s, \ell}  \Bigg)\\
+ & 2\sum_{0\le i,j\le n-1} \varphi\Bigg( \sum_{r=0}^{n-1}\ \sum_{s=0}^{n-1}
t_{j-i,r-i}t_{r, s}t_{s, \ell} + \sum_{r=0}^{n-1}
t_{j-i,r-i}t_{r, \ell}\Bigg),\\
+ & \sum_{0\le i,j\le n-1} \varphi\Bigg( \sum_{r=0}^{n-1}\ \sum_{s=0}^{n-1}\sum_{k=0}^{n-1}
t_{j-i,r-i}t_{r, s} t_{s, k}t_{k, \ell} \\
+ & \sum_{r=0}^{n-1}\ \sum_{s=0}^{n-1}
t_{j-i,r-i}t_{r, s} t_{s, \ell}+ \sum_{s=0}^{n-1}
t_{j-i,r-i}t_{r, \ell}\Bigg),
\end{array}$$
where subscripts are taken modulo $n$,
$(t_{i,j})_{0\le i,j\le n-1}$ is the multiplication table
of $\mathcal{N}$, and 
$\varphi$ is the unique ring homomorphism
from $\bF_{\!2}$ into  $\mathbb{Z}$.

\item For $n\le \ell\le 2n-1$, the number of non-zero entries in the
 $\ell$-th multiplication table of $\mathcal{W}$, is equal to
$$\begin{array}{rl}
 & 4w(\mathcal{N})+
 \sum_{0\le i,j\le n-1} \varphi \Bigg( \sum_{r=0}^{n-1} t_{j-i,r-i}t_{r, \ell}
 + t_{j-i,\ell-i}  \Bigg)\\
+ & 2 \sum_{0\le i,j\le n-1} \varphi\Bigg( \sum_{r=0}^{n-1}\ \sum_{s=0}^{n-1}
t_{j-i,r-i}t_{r, s} t_{s, \ell} + \sum_{r=0}^{n-1}
t_{j-i,r-i}t_{r, \ell}+ t_{j-i,\ell-i}\Bigg).
\end{array}$$
 
\item For $2n\le \ell\le 3n-1$, the number of non-zero entries in the
 $\ell$-th multiplication table of $\mathcal{W}$, is equal to
$$3w(\mathcal{N})+3\sum_{0\le i,j\le n-1} \sum_{r=0}^{n-1}\varphi(t_{j-i,r-i}t_{r,\ell}).$$

\item The $\ell$-th multiplication table of $\mathcal{W}$, for $3n\le \ell\le 4n-1$, 
 has $9w(\mathcal{N})$ non-zeros entries.
\end{enumerate}
\end{lemma}

\begin{proof}
The entries of the multiplication tables $T_\ell$ of
$\mathcal{W}$ are given by the block matrix

\vspace{.25cm}

$$\left( \begin{array}{c|c|c|c}
& & & \\ 
(\alpha^{2^i}\alpha^{2^j})_{0\le i,j\le n-1} & (\beta_0\alpha^{2^i}\alpha^{2^j})_{0\le i,j\le n-1}
& (\beta_1\alpha^{2^i}\alpha^{2^j})_{0\le i,j\le n-1} & 
(\beta_0\beta_1\alpha^{2^i}\alpha^{2^j})_{0\le i,j\le n-1} \\ 
& & & \\ \hline
& & & \\ 
(\beta_0\alpha^{2^i}\alpha^{2^j})_{0\le i,j\le n-1} & (\beta_0^2\alpha^{2^i}\alpha^{2^j})_{0\le i,j\le n-1}
& (\beta_0\beta_1\alpha^{2^i}\alpha^{2^j})_{0\le i,j\le n-1} &
 (\beta_0^2\beta_1\alpha^{2^i}\alpha^{2^j})_{0\le i,j\le n-1} \\ 
& & & \\ \hline
& & & \\ 
(\beta_1\alpha^{2^i}\alpha^{2^j})_{0\le i,j\le n-1} & (\beta_0 \beta_1\alpha^{2^i}\alpha^{2^j})_{0\le i,j\le n-1}
& (\beta_1^2\alpha^{2^i}\alpha^{2^j})_{0\le i,j\le n-1} &
 (\beta_0\beta_1^2\alpha^{2^i}\alpha^{2^j})_{0\le i,j\le n-1} \\ 
& & & \\ \hline
& & & \\ 
(\beta_0 \beta_1\alpha^{2^i}\alpha^{2^j})_{0\le i,j\le n-1} &
 (\beta_0^2\beta_1\alpha^{2^i}\alpha^{2^j})_{0\le i,j\le n-1}
& (\beta_0\beta_1^2\alpha^{2^i}\alpha^{2^j})_{0\le i,j\le n-1}
& (\beta_0^2\beta_1^2\alpha^{2^i}\alpha^{2^j})_{0\le i,j\le n-1} \\ 
& & & \\
\end{array} \right)
$$

Recall that equation $(\ref{eq:12})$ yields
$\beta_0^2=\beta_0+ \alpha \ \text{ and } \
\beta_1^2= \beta_1+\beta_0(\alpha+1)+\alpha^2.$
Hence:
\begin{enumerate}[1.]
\item For $0\le \ell\le n-1$, the components of the matrix
$T_\ell$ come from 1 block $(\alpha^{2^i}\alpha^{2^j})_{0\le i,j\le n-1} $,
1 block $(\beta_0^2\alpha^{2^i}\alpha^{2^j})_{0\le i,j\le n-1} $,
1 block $(\beta_1^2\alpha^{2^i}\alpha^{2^j})_{0\le i,j\le n-1} $,
2 blocks $(\beta_0\beta_1^2\alpha^{2^i}\alpha^{2^j})_{0\le i,j\le n-1} $ and 
1 block $(\beta_0^2\beta_1^2\alpha^{2^i}\alpha^{2^j})_{0\le i,j\le n-1}$.
These correspond to
1 block $(\alpha^{2^i}\alpha^{2^j})_{0\le i,j\le n-1} $,
1 block $(\alpha\alpha^{2^i}\alpha^{2^j})_{0\le i,j\le n-1} $,
1 block $(\alpha^2\alpha^{2^i}\alpha^{2^j})_{0\le i,j\le n-1} $,
2 blocks $((\alpha^2+\alpha)\alpha^{2^i}\alpha^{2^j})_{0\le i,j\le n-1}$ and 
1 block $((\alpha^3+\alpha^2+\alpha)\alpha^{2^i}\alpha^{2^j})_{0\le i,j\le n-1}$.
From the study made in the proof of Proposition \ref{prop:1} (see equation $(\ref{eq:8})$),
we know that
$$
\alpha^{2^i}\alpha^{2^j}=
\sum_{\ell=0}^{n-1}t_{j-i,\ell-i}\alpha^{2^\ell},
$$
where subscripts are taken modulo $n$ and
$(t_{i,j})_{0\le i,j\le n-1}$ is the multiplication table
of $\mathcal{N}$. So
$$
\alpha \alpha^{2^i}\alpha^{2^j}=
\sum_{\ell=0}^{n-1}\Bigg(\sum_{r=0}^{n-1}t_{j-i,r-i}t_{r,\ell}\Bigg)\alpha^{2^\ell},
\quad
\alpha^2 \alpha^{2^i}\alpha^{2^j}=\sum_{\ell=0}^{n-1}
\Bigg(\sum_{s=0}^{n-1}\sum_{r=0}^{n-1}t_{j-i,r-i}t_{r,s}t_{s, \ell}\Bigg)\alpha^{2^\ell},
$$
and
$$
\alpha^3 \alpha^{2^i}\alpha^{2^j}=\sum_{\ell=0}^{n-1}
\sum_{k=0}^{n-1}\sum_{s=0}^{n-1}\sum_{r=0}^{n-1}t_{j-i,r-i}t_{r,s}t_{s, k}t_{k, \ell}
 \alpha^{2^\ell}.
$$
We conclude that the number of non-zero entries in 
$T_\ell$ is equal to
$$
\begin{array}{rl}
 & w(\mathcal{N})+\sum_{0\le i,j\le n-1} \varphi( \sum_{r=0}^{n-1}t_{j-i,r-i}t_{r,\ell})\\
+ & \sum_{0\le i,j\le n-1} \varphi \Bigg(
 \sum_{r=0}^{n-1}\ \sum_{s=0}^{n-1}t_{j-i,r-i}t_{r, s} t_{s, \ell}  \Bigg)\\
+ & 2\sum_{0\le i,j\le n-1} \varphi\Bigg( \sum_{r=0}^{n-1}\ \sum_{s=0}^{n-1}
t_{j-i,r-i}t_{r, s}t_{s, \ell} + \sum_{r=0}^{n-1}
t_{j-i,r-i}t_{r, \ell}\Bigg),\\
+ & \sum_{0\le i,j\le n-1} \varphi\Bigg( \sum_{r=0}^{n-1}\ \sum_{s=0}^{n-1}\sum_{k=0}^{n-1}
t_{j-i,r-i}t_{r, s} t_{s, k}t_{k, \ell} \\
+ & \sum_{r=0}^{n-1}\ \sum_{s=0}^{n-1}
t_{j-i,r-i}t_{r, s} t_{s, \ell}+ \sum_{s=0}^{n-1}
t_{j-i,r-i}t_{r, \ell}\Bigg),
\end{array}$$
where $\varphi$ is the unique ring homomorphism
from $\bF_{\!2}$ into  $\mathbb{Z}$.

\item For $n\le \ell\le 2n-1$, the components of the matrix
$T_\ell$ come from 2 blocks $(\beta_0\alpha^{2^i}\alpha^{2^j})_{0\le i,j\le n-1} $,
1 block $(\beta_0^2\alpha^{2^i}\alpha^{2^j})_{0\le i,j\le n-1} $,
1 block $(\beta_1^2\alpha^{2^i}\alpha^{2^j})_{0\le i,j\le n-1} $,
2 blocks $(\beta_0\beta_1^2\alpha^{2^i}\alpha^{2^j})_{0\le i,j\le n-1} $ and 
1 block $(\beta_0^2\beta_1^2\alpha^{2^i}\alpha^{2^j})_{0\le i,j\le n-1}$.
These correspond to  4 blocks
$(\beta_0\alpha^{2^i}\alpha^{2^j})_{0\le i,j\le n-1} $,
1 block $((\alpha+1)\beta_0\alpha^{2^i}\alpha^{2^j})_{0\le i,j\le n-1}$ and 
2 blocks $((\alpha^2+\alpha+1)\beta_0\alpha^{2^i}\alpha^{2^j})_{0\le i,j\le n-1}$.
So the number of non-zero entries in 
$T_\ell$ is equal to
$$\begin{array}{rl}
 & 4w(\mathcal{N})+
 \sum_{0\le i,j\le n-1} \varphi \Bigg( \sum_{r=0}^{n-1} t_{j-i,r-i}t_{r, \ell}
 + t_{j-i,\ell-i}  \Bigg)\\
+ & 2 \sum_{0\le i,j\le n-1} \varphi\Bigg( \sum_{r=0}^{n-1}\ \sum_{s=0}^{n-1}
t_{j-i,r-i}t_{r, s} t_{s, \ell} + \sum_{r=0}^{n-1}
t_{j-i,r-i}t_{r, \ell}+ t_{j-i,\ell-i}\Bigg).
\end{array}$$
 
\item For $2n\le \ell\le 3n-1$, the components of the matrix
$T_\ell$ come from 2 blocks $(\beta_1\alpha^{2^i}\alpha^{2^j})_{0\le i,j\le n-1} $,
2 blocks $(\beta_0^2\beta_1\alpha^{2^i}\alpha^{2^j})_{0\le i,j\le n-1} $,
1 block $(\beta_1^2\alpha^{2^i}\alpha^{2^j})_{0\le i,j\le n-1} $, and 
1 block $(\beta_0^2\beta_1^2\alpha^{2^i}\alpha^{2^j})_{0\le i,j\le n-1}$.
These correspond to 3 blocks
$(\beta_1\alpha^{2^i}\alpha^{2^j})_{0\le i,j\le n-1} $,
3 blocks $(\beta_1\alpha\alpha^{2^i}\alpha^{2^j})_{0\le i,j\le n-1}$.
So the number of non-zero entries in 
$T_\ell$ is equal to
$$3w(\mathcal{N})+3\sum_{0\le i,j\le n-1} \varphi(\sum_{r=0}^{n-1}t_{j-i,r-i}t_{r,\ell}).$$

\item For $3n\le \ell\le 4n-1$, the components of the matrix
$T_\ell$ come from 4 blocks $(\beta_0\beta_1\alpha^{2^i}\alpha^{2^j})_{0\le i,j\le n-1} $,
2 blocks $(\beta_0^2\beta_1\alpha^{2^i}\alpha^{2^j})_{0\le i,j\le n-1} $,
2 blocks $(\beta_0\beta_1^2\alpha^{2^i}\alpha^{2^j})_{0\le i,j\le n-1} $ and 
1 block $(\beta_0^2\beta_1^2\alpha^{2^i}\alpha^{2^j})_{0\le i,j\le n-1}$.
This means that we have 9 blocks
$(\beta_0\beta_1\alpha^{2^i}\alpha^{2^j})_{0\le i,j\le n-1} $.
So the number of non-zero entries in 
$T_\ell$ is equal to $9w(\mathcal{N}).$
\end{enumerate}
\end{proof}


\section{Kummer extended bases with degree prime to $2$}

Cyclic extensions of $\bK$ with degree
prime to $p$ are described by Kummer theory.
Indeed, let $n\ge 2$ be a  prime to $p$ integer
such that $\bK$ contains a primitive $n$-root of unity.
It is proved 
[\cite{Lang}, Chapter VI, Theorem 6.2] that every degree $n$ cyclic extension
$\bL$ of $\bK$ is generated by a radical. This means that there exists
a non-zero element $a$ in $\bK$ whose class in 
$\bK^*/\bK^{*n}$ has order $n$ and such that $\bL$ is isomorphic to $\bK[X]/(X^n-a).$
However, irreducible polynomials of the form $X^n-a$
may also be used for extending normal bases.

\begin{definition}
Let $p$ be a prime number and $q$ a power of $p$.
Let $\mathcal{N} =(\alpha^{q^{i}})_{0\le i\le n-1}$ be a normal basis of 
$\mathbf{F}_{\!\!q^{n}}/\mathbf{F}_{\!\!q}$. Denote by $\overline{\bF}_{\!\!q}$
an algebraic closure of $\bF_{\!\!q}$ containing $\mathbf{F}_{\!\!q^{n}}$.
Assume that $\mathbf{F}_{\!\!q^{n}}$ possesses a primitive $d$-th root of unity.
A degree $d$ Kummer extension (also Kummer extended basis) of $\mathcal{N}$ is a basis
$\mathcal{K}$ of $\mathbf{F}_{\!\!q^{nd}}/\mathbf{F}_{\!\!q}$
for which there exists $\beta \in \overline{\bF}_{\!\!q}$
outside of $\mathbf{F}_{\!\!q^n}$ such that
$\beta^d-\alpha=0$ and  $\mathcal{K}=(\alpha^{q^i}\beta^j)_{i,j}$.
\end{definition}

In this section, we are interested in degree $3$ Kummer extensions of normal bases 
of $\mathbf{F}_{\!2^{n}}/\mathbf{F}_2$.

\subsection{Complexity of degree $3$ Kummer extended bases in characteristic $2$}\label{sect:3}

In general, a normal basis $\{\alpha,\alpha^q,\ldots,\alpha^{q^{n-1}}\}$
of $\mathbf{F}_{\!q^{n}}/\mathbf{F}_q$ is said to be {\it primitive} if
$\alpha$ generates the multiplicative
group $\bF_{\!q^n}^*$.  So any
primitive normal basis
of $\mathbf{F}_{\!q^{n}}/\mathbf{F}_q$
admits a degree $d$ Kummer extension, provided that $d$ divides $q^n-1$.
Lenstra and Schoof \cite{Lenstra-Schoof} showed that for any prime power $q$
and positive integer $n$, there is a primitive normal
basis of $\mathbf{F}_{\!q^n}$ 
over $\mathbf{F}_q$. The following proposition describes
degree $3$ Kummer extensions of primitive normal bases
of $\mathbf{F}_{\!2^{n}}/\mathbf{F}_{\!2}$.

\begin{prop}\label{prop:3}
Let $n$ be a positive integer such that $3$ divides $2^n-1$. 
Assume that $\mathcal{N}=
(\alpha^{2^{i}})_{1\le i\le n-1}$ is a primitive normal basis 
$\mathbf{F}_{\!2^{n}}/\mathbf{F}_{\!2}$. Then:
\begin{enumerate}[1.]
\item There exists $\beta$ in $\bF_{\!2^{3n}}$ such that
$\mathcal{K}= \mathcal{N} \cup \beta\mathcal{N}\cup \beta^2\mathcal{N}$
is a degree $3$ Kummer extension of $\mathcal{N}$.

\item If $\gamma=C+\beta D+\beta^2 E$ is an element of $\bF_{\!2^{3n}}$
expressed in $\mathcal{K}$, then squaring is given by
$$\gamma^2 = C_{>}+\beta G+\beta^2 D_{>},$$
where $C_{>}$ and $D_{>}$ stand for right-cyclic shifts of the coordinate vectors
of $C$ and $D$; and 
 $$G = \  ^t\!E_{>} \times T$$
is a vector-matrix multiplication between the transpose of 
the right-cyclic shift of the coordinate vector
of $E$ and the multiplication table of \ $\mathcal{N}$.

\item The complexity of $\mathcal{K}$ consists in at 
most:
\begin{enumerate}
\item $6$ multiplications and $15$ additions between elements of $\bF_{\!2^{n}}$,
\item $2$ vector-matrix 
multiplications between vectors in $\bF_{\!2^n}$ and
the multiplication table of $\mathcal{N}$.
\end{enumerate}

\item  If $\mathcal{N}$ has subquadratic complexity and subquadratic weight in $n$,
then $\mathcal{K}$ has also subquadratic complexity in $n$.
\end{enumerate}
\end{prop}

\begin{proof}
\begin{enumerate}[1.]
\item Since $\alpha$ generates $\bF_{\!2^{n}}^*$,
the polynomial $x^3-\alpha$ is irreducible over $\bF_{\!2^{n}}$.
The result follows from [\cite{Lang}, Chapter VI, Theorem 6.2].

\item Let $C=\sum_{i=0}^{n-1}c_i\alpha^{2^i}$, 
$D=\sum_{i=0}^{n-1}d_i\alpha^{2^i}$ and $E=\sum_{i=0}^{n-1}e_i\alpha^{2^i}$
 be the linear combinations of $C,D$ and $E$
with respect to $\mathcal{N}$.
We have
$$
\begin{array}{rl}
(C+\beta D+\beta^2 E)^2 = & \sum_{i=0}^{n-1}c_{i-1}\alpha^{2^i} + \beta^2\sum_{i=0}^{n-1}d_{i-1}\alpha^{2^i}
+ \beta^4\sum_{i=0}^{n-1}e_{i-1}\alpha^{2^i}\\
= & \sum_{i=0}^{n-1}c_{i-1}\alpha^{2^i} +\beta\sum_{i=0}^{n-1}e_{i-1}\sum_{k=0}^{n-1}
t_{ik}\alpha^{2^k} + \beta^2 \sum_{i=0}^{n-1}d_{i-1}\alpha^{2^i}.
\end{array}
$$
where subscripts are taken modulo $n$ and $(t_{ik})_{i,k}$ stands for the multiplication 
table of $\mathcal{N}$.
So $$(C+\beta D+\beta^2 E)^2 = C_{>}+\beta (^t\!E_{>} \times T)+\beta^2 D_{>},$$
where $C_{>},D_{>}$ and $E_{>}$ stand for right-cyclic shifts of the coordinate vectors
of \ $C,D,E$, and 
 $$  ^t\!E_{>} \times T$$
is a vector-matrix multiplication between the transpose of \ $E_{>} $
and \ the multiplication table of \ $\mathcal{N}$.

\item Let $C=C_0+\beta C_1+\beta^2 C_2$ and $D=D_0+\beta D_1
+\beta^2 D_2$ be two elements
of \  $\bF_{\!2^{3n}}$ expressed in $\mathcal{K}$.
A Karatsuba-like multiplication algorithm gives
\small{$$
\begin{array}{rrl}
C \times D & = & \beta^4C_2D_2 +\beta^2 \Big( (C_2+\beta C_1+C_0)(D_2+\beta D_1+D_0)
+ C_2D_2 + (\beta C_1+C_0)(\beta D_1 + D_0)\Big)\\
 &  & + (\beta C_1+C_0)(\beta D_1 + D_0).
\end{array}$$}
So
$$
\begin{array}{rrl}
C \times D & = & \beta^4C_2D_2 +\beta^2
\Bigg(  \beta \Big( (C_0+C_1+C_2)(D_0+D_1+D_2)
+(C_0+C_2)(D_0+D_2)\\
 & &+(C_0+C_1)(D_0+D_1)+C_0D_0 \Big) 
 + C_2D_2+(C_0+C_2)(D_0+D_2)+C_0D_0 \Bigg)\\
 & & + \beta^2 C_1D_1 +\beta \Big( (C_0+C_1)(D_0+D_1) + C_1D_1+C_0D_0\Big)
 +C_0D_0.
\end{array}
$$

Since $\beta^3=\alpha$, we have
\begin{equation}\label{eq:10}
\begin{array}{rrl}
C \times D & = & C_0D_0+\alpha \Big( C_0D_0 +
(C_0+C_1)(D_0+D_1) + (C_0+C_2)(D_0+D_2)\\
 & & +(C_0+C_1+C_2)(D_0+D_1+D_2)\Big)\\
 & & +\beta \big(C_0D_0+C_1D_1+ \alpha C_2D_2 + (C_0+C_1)(D_0+D_1)
\big) \\
 & & + \beta^2 \Big(C_0D_0+ C_1D_1+C_2D_2+(C_0+C_2)(D_0+D_2)\Big) \end{array}
\end{equation}
So the product $C \times D$ consists in:
\begin{enumerate}
\item  6 products and 15 additions
between elements lying in the field $\bF_{\!2^{n}}$;

\item 2 vector-matrix 
multiplications between vectors in $\bF_{\!2}^n$ and
the multiplication table of $\mathcal{N}$. These correpond to
the computation of the terms $\alpha C_2D_2$ and 
$$
\alpha \Big( C_0D_0 + (C_0+C_1)(D_0+D_1) + (C_0+C_2)(D_0+D_2)+(C_0+C_1+
C_2)(D_0+D_1+D_2)\Big).
$$
\end{enumerate}

\item The same argument as in the proof of
Proposition \ref{prop:1} shows that a normal basis with subquadratic weight and
subquadratic complexity in $n$ yields Kummer extended bases with subquadratic
complexity in $n$.
\end{enumerate}
\end{proof}

\subsection{Density}

We just described squaring and multiplication in 
$\bF_{\!2^{3n}}$ with respect to a Kummer extension
$\mathcal{K}= \mathcal{N} \cup \beta\mathcal{N}\cup \beta^2\mathcal{N}$
of a primitive normal basis
$$\mathcal{N}=\{\alpha,\alpha^2,\ldots,\alpha^{2^{n-1}}\}$$
of \ $\bF_{\!2^{n}}/\bF_{\!2}$.
In this section we are interested in multiplication tables of 
$\mathcal{K}$. These are
$3n\times 3n$ matrices with entries in $\bF_{\!2}$.
For $(i,\delta)\in \{0,\ldots, n-1\} \times \{0,1,2 \}$,
we set $\kappa_{i+\delta n}=\alpha^{2^i}\beta^\delta$ so that 
$\mathcal{K}=(\kappa_\ell)_{0\le \ell\le 3n-1}$.

\begin{lemma}\label{lem:1}
With the above notation, 
let $w(\mathcal{N})$ be the weight of the normal basis $\mathcal{N}$.
\begin{enumerate}[1.]
\item For $0\le \ell\le n-1$, the number of non-zero entries in the
 $\ell$-th multiplication table of $\mathcal{K}$, is equal to
$$w(\mathcal{N})+2 \sum_{0\le i,j\le n-1}
\varphi(\sum_{r=0}^{n-1}t_{j-i,\ell-i}t_{r,\ell}),$$
where subscripts are taken modulo $n$,
$(t_{i,j})_{0\le i,j\le n-1}$ is the multiplication table
of $\mathcal{N}$, and 
$\varphi$ is the unique ring homomorphism
from $\bF_{\!2}$ into  $\mathbb{Z}$.

\item For $n\le \ell\le 2n-1$, the number of non-zero entries in the
 $\ell$-th multiplication table of $\mathcal{K}$, is equal to
$$2w(\mathcal{N})+\sum_{0\le i,j\le n-1}
\varphi(\sum_{r=0}^{n-1}t_{j-i,\ell-i}t_{r,\ell}).$$
 
\item The $\ell$-th multiplication table of $\mathcal{K}$, for $2n\le \ell\le 3n-1$, 
 has $3w(\mathcal{N})$ non-zeros entries.
\end{enumerate}
The density of $\mathcal{K}$ is given by
$$d(\mathcal{K})=6d(\mathcal{N})
+ 3 \sum_{0\le \ell\le n-1} \sum_{0\le i,j\le n-1}
\varphi(\sum_{0\le r \le n-1}t_{j-i,r-i}t_{r,\ell}).$$
\end{lemma}

\begin{proof}
The multiplication tables
$T_\ell$ of $\mathcal{K}$ are given by the block matrix
$$\left( \begin{array}{c|c|c}
& & \\ 
(\alpha^{2^i}\alpha^{2^j})_{0\le i,j\le n-1} & (\beta\alpha^{2^i}\alpha^{2^j})_{0\le i,j\le n-1}
& (\beta^2\alpha^{2^i}\alpha^{2^j})_{0\le i,j\le n-1}\\ 
& & \\ \hline
& & \\ 
(\beta \alpha^{2^i}\alpha^{2^j})_{0\le i,j\le n-1} & (\beta^2\alpha^{2^i}\alpha^{2^j})_{0\le i,j\le n-1}
& (\alpha\alpha^{2^i}\alpha^{2^j})_{0\le i,j\le n-1}\\ 
& & \\ \hline
& & \\ 
(\beta^2\alpha^{2^i}\alpha^{2^j})_{0\le i,j\le n-1} & (\alpha\alpha^{2^j}\alpha_j)_{0\le i,j\le n-1}
& (\beta\alpha\alpha^{2^i}\alpha^{2^j})_{0\le i,j\le n-1}\\
& & \\
\end{array} \right)
$$
\begin{enumerate}[1.] 
\item For \ $0\le \ell\le n-1$, the components of the matrix
$T_\ell$ come from 1 block $(\alpha^{2^i}\alpha^{2^j})_{0\le i,j\le n-1}$ and 
2 blocks $(\alpha \alpha^{2^i}\alpha^{2^j})_{0\le i,j\le n-1}$. 
Using the same argument as in the proof of Proposition \ref{prop:1},
we conclude that the total number of non-zero entries in $T_\ell$ is equal to
$$w(\mathcal{N})+2 \sum_{0\le i,j\le n-1}
\varphi(\sum_{r=0}^{n-1}t_{j-i,r-i}t_{r,\ell}),$$
where subscripts are taken modulo $n$,
$(t_{i,j})_{0\le i,j\le n-1}$ is the multiplication table
of $\mathcal{N}$, and 
$\varphi$ is the unique ring homomorphism
from $\bF_{\!2}$ into  $\mathbb{Z}$.
\item For \ $n\le \ell\le 2n-1$, the components of the matrix
$T_\ell$ come from 2 blocks $(\beta\alpha_i\alpha_j)_{0\le i,j\le n-1}$ and 
1 block $(\beta\alpha \alpha_i\alpha_j)_{0\le i,j\le n-1}$.
So the total number of non-zero entries in $T_k$ is equal to
$$2w(\mathcal{N})+ \sum_{0\le i,j\le n-1}
\varphi(\sum_{r=0}^{n-1}t_{j-i,r-i}t_{r,\ell}).$$

\item For \ $2n\le \ell\le 3n-1$, the components of the matrix
$T_\ell$ come from 3 blocks $(\beta^2\alpha_i\alpha_j)_{0\le i,j\le n-1}$.
So the total number of non-zero entries in $T_\ell$ is equal to $3w(\mathcal{N})$.
\end{enumerate}
\end{proof}

\begin{table}
\caption{Sums of cross-products of the multiplication table of the best known
normal bases of $\bF_{\!2^n}/\bF_{\!2}$, for even integers $2\le n\le 14$}\label{Tableau2}
$\begin{array}{|c|c|c|c|}\hline
n & \text{Modulus} & \text{Normal elements}& \sum_{0\le \ell\le n-1} \sum_{0\le i,j\le n-1}
\varphi(\sum_{0\le r \le n-1}t_{j-i,r-i}t_{r,\ell})\\ \hline
2&1+x+x^2& x& 5\\ \hline
4&1+x+x^4&x^3& 25 \\ \hline
6&1+x+x^6& x^3+x^4+x^5&101 \\ \hline
8&1+x+x^3+x^4+x^8&x^6+x^7& 233\\ \hline
10&1+x^3+x^{10} & x^3+x^5+x^7+x^9 &181\\ \hline
12&1+x^3+x^{12}&  {\begin{array}{l} x^2+x^3+x^4+x^5\\
+x^6+x^7+x^8+x^9 \end{array}} & 265 \\ \hline
14& 1+x^5+x^{14}& {\begin{array}{l} x^5+x^6+x^7+x^9\\
+x^{12}+x^{13}\end{array}} & 677 \\ \hline
\end{array}
$
\end{table}

\begin{table}
\caption{Sums of cross-products of the multiplication table of the best known
normal bases of $\bF_{\!2^n}/\bF_{\!2}$, for even integers $16\le n\le 26$}\label{Tableau3}
$\begin{array}{|c|c|c|c|c|}\hline
n & \text{Modulus} & \text{Normal elements}& \sum_{0\le \ell\le n-1} \sum_{0\le i,j\le n-1}
\varphi(\sum_{0\le r \le n-1}t_{j-i,r-i}t_{r,\ell})\\ \hline
16& 1+x^3+x^{16}+x^{16}& {\begin{array}{l}  x^6+x^8+x^9+x^{11}+x^{12}\\
+x^{13}+x^{14}+x^{15}\end{array}} & 1921\\ \hline
18& 1+x^3+x^{18}& {\begin{array}{l}  x^4+x^5+x^7+x^8+x^9\\
+x^{11}+x^{15}+x^{16}+x^{17} \end{array}} & 613 \\ \hline
20&  1+x^3+x^{20}& {\begin{array}{l}  x^3+x^8+x^{11}+x^{15}+x^{16}\\
+x^{17}+x^{18}+x^{19} \end{array}} & 1625 \\ \hline
22&1+x+x^{22}& {\begin{array}{l} x^8+x^{11}+x^{12}\\
+x^{19}+x^{20}+x^{21} \end{array}}  & 2005 \\ \hline
24& 1+x+x^3+x^4+x^{24} & {\begin{array}{l}
x^5+x^6+x^{10}+x^{16}\\
+x^{17}+x^{18}+x^{19}+x^{23} \end{array}} & 3961\\ \hline
26&1+x+x^3+x^4+x^{26} &{\begin{array}{l}
 x^5+x^{10}+x^{12}+x^{15}+x^{16}\\
+x^{19} +x^{20}+x^{21}+x^{22}\\
 +x^{23}+x^{25}\end{array}} &2501
 \\ \hline
\end{array}
$
\end{table}

We computed the sums 
$$\sum_{\ell=0}^{n-1}\sum_{0\le i,j\le n-1}
\varphi(\sum_{r=0}^{n-1}t_{j-i,r-i}t_{r,\ell})$$
of cross-products of the multiplication table of the best known
normal bases of $\ \mathbf{F}_{\!2^n}/\bF_{\!2}$ for even integers $2\le n\le 26$.
The results are provided by tables \ref{Tableau2} and \ref{Tableau3}. 
These sums are useful when computing densities of Kummer extended bases from formula
given in Lemma \ref{lem:1}. To design the tables we used
[\cite{Mullen-Panario13}, Section 2.2] and the website accompanying it
which is available at 
\href{https://people.math.carleton.ca/~daniel/hff/}{https://people.math.carleton.ca/~daniel/hff/}.


\section{Towers of extensions}

It is clear that extended bases obtained by iterating Artin-Schreier theory
corresponds to extended bases constructed from Artin-Schreier-Witt theory. 
In this section, we study extended bases in the context of towers
of field extensions constructed from Kummer theory.
We are also interested in towers
combining Artin-Schreier and Kummer theories.
Indeed any primitive normal basis of $\ \mathbf{F}_{\!2^n}/\bF_{\!2}$
admits a Kummer extension
of degree  $d$, provided $d$ divides $2^n-1$.
A question is whether the Kummer extended basis itself admits a 
Kummer extension or an Artin-Schreier extension.

\begin{lemma}\label{lem:2}
Let $\mathcal{N}=(\alpha, \alpha^{2}, \ldots,\alpha^{2^{n-1}})$
be a normal basis of \ $\bF_{\!2^n}/\bF_{\!2}$.

\begin{enumerate}[1.]
\item There exists $\beta$ in $\bF_{\!2^{2n}}$ such that 
$\mathcal{A}=\mathcal{N}\cup \beta \mathcal{N}$ is an Artin-Schreier 
extension of $\mathcal{N}$.
\begin{enumerate}
\item The polynomial $X^2+X+\beta$
is irreducible over $\bF_{\!2^{2n}}$ if and only if
 $n$ is odd (if that is the case one says that $\mathcal{N}$
admits a degree $4$ Artin-Schreier-Witt
extension, or a {\it biquadratic Artin-Schreier extension}).
\item Assume that $3$ divides $2^{2n}-1$.
Then the polynomial $X^3+\beta$ is irreducible over $\bF_{\!2^{2n}}$
if and only if the class of $\beta$ generates
$\bF_{\!2^{2n}}^*/\bF_{\!2^{2n}}^{*3}$
(if that is the case one says that the Artin-Schreier extension $\mathcal{A}$
admits a degree $3$ Kummer extension).
\end{enumerate}

\item Assume that $3$ divides $2^{n}-1$ and that
$\beta$ is an element  in $\bF_{\!2^{3n}}$ such that
\ $\mathcal{K}=\mathcal{N}\cup \beta\mathcal{N}\cup \beta^2\mathcal{N}$
is a degree $3$ Kummer extension of $\mathcal{N}$. Then:
\begin{enumerate}
\item The polynomial $X^2+X+\beta$ is always reducible over $\bF_{\!2^{3n}}$
(one says that the Kummer extension $\mathcal{K}$
admits no Artin-Schreier extension).

\item If $\mathcal{N}$ is a primitive normal basis, and
if the $3$-adic valuation satisfies $$v_3\Big(\frac{2^{3n}-1}{2^n-1}\Big)=1,$$
then the polynomial $X^3+\beta$ is irreducible over $\bF_{\!2^{3n}}$
(in that case one says that $\mathcal{N}$
admits a { \it bicubic Kummer extension}). 
\end{enumerate}
\end{enumerate}
\end{lemma}

\begin{proof}
\begin{enumerate}[1.]
\item This is [\cite{Thomson-Weir}, Lemma 3.4].
\begin{enumerate}
\item This assertion corresponds [\cite{Thomson-Weir}, Lemma 5.1].
\item The assertion follows from
[\cite{Lang}, Chapter VI, Theorem 6.2] or [\cite{Bourbakialg4.5}, A V.84].
\end{enumerate}

\item 
\begin{enumerate}
\item The characteristic  polynomial of $\beta$ over $\bF_{\!2^{n}}$  is
$X^3-\alpha \in \bF_{\!2^{n}}[X].$
So $\mathrm{Tr}_{\bF_{\!2^{3n}}/\bF_{\!2^{n}}}(\beta)=0$.
We have
$$
\mathrm{Tr}_{\bF_{\!2^{3n}}/\bF_{\!2}}(\beta)=
\mathrm{Tr}_{\bF_{\!2^{n}}/\bF_{\!2}}(\mathrm{Tr}_{\bF_{\!2^{3n}}/\bF_{\!2^n}}(\beta))
=0.$$
From [\cite{Lang}, Chapter VI, Theorem 6.3],
there exists $\gamma$ in
$\bF_{\!2^{3n}}$ such that $\beta=\gamma^2+\gamma$. So
$X^2+X+\beta$ is a reducible polynomial
over $\bF_{\!2^{3n}}$.

\item We know that $\beta^3=\alpha$. So
$\beta$ has order $3(2^n-1)$ in 
$\bF_{\!2^{3n}}^*$ because $3$ divides $2^n-1$ and
$\alpha$ generates $\bF_{\!2^{n}}^*$.
Let $\delta$ be a generator of $\bF_{\!2^{3n}}^*$.
Then there exists an integer $r\ge 1$ which is prime to $2^{3n}-1$ such that
$$\beta=\delta^{\frac{r(2^{3n}-1)}{3(2^n-1)}}.$$
Since $v_3\Big(\frac{2^{3n}-1}{2^n-1}\Big)=1$ and $r$ is
prime to $3$, we have
$$v_3\Big(\frac{r(2^{3n}-1)}{3(2^n-1)}\Big)=0.$$
So $\beta$ is not a cube in $\mathbf{\bF}_{\!2^{3n}}$.
We conclude that $\bF_{\!2^{3n}}(\beta)$ is a degree $3$ cyclic extension
of $\bF_{\!2^{3n}}$ by [\cite{Lang}, Chapter VI, Theorem 6.2] or [\cite{Bourbakialg4.5}, A V.84].
\end{enumerate}
\end{enumerate}
\end{proof}

\section{Conclusion}

This paper presents bases of $\bF_{\!2^{nd}}/\bF_{\!2}$
constructed by extending normal bases of $\bF_{\!2^n}/\bF_{\!2}$
from Artin-Schreier theory and Kummer theory respectively.
In case $d$ is equal to $2$, $3$ and $4$, we explain how squaring in
$\bF_{\!2^{nd}}$ can be efficiently computed from the extended bases.
We also explain how a Karatsuba-like multiplication algorithm 
may be used to efficiently compute the product of two elements
in $\bF_{\!2^{nd}}$.
Then we specify conditions under which Artin-Schreier and Kummer
theories may be combined in order to extend normal bases of $\bF_{\!2^n}/\bF_{\!2}$.

From the study made in Sections
2 and 3, we can actually determine properties of Kummer extensions
of an Artin-Schreier extended basis.
Indeed let $\mathcal{N}=(\alpha, \alpha^{2}, \ldots,\alpha^{2^{n-1}})$ be
a normal basis of \ $\bF_{\!2^n}/\bF_{\!2}$ and
$$\mathcal{A}=\mathcal{N}\cup \beta \mathcal{N}$$
an Artin-Schreier extension of $\mathcal{N}$. 
Assume that $3$ divides $2^{2n}-1$. Assume that $\beta$
is not a cube in $\mathbf{\bF}_{\!2^{2n}}$.
Let $\gamma$ be an
element of $\bF_{\!2^{6n}}$ such that 
$$\mathcal{K}\!\mathcal{A}=\mathcal{A}\cup \gamma \mathcal{A}\cup \gamma^2 \mathcal{A}$$
is a degree $3$ Kummer extension of $\mathcal{A}$ (see Lemma \ref{lem:2} $1.(b)$).
Multiplications in $\bF_{\!2^{6n}}$ with respect to
$\mathcal{K}$ is described from Propositions \ref{prop:1}
and \ref{prop:3}. On the one hand, squaring an element 
$$X=(A+\beta B)+\gamma (C+\beta D)
+\gamma^2(E+\beta F) \in \bF_{\!2^{6n}}$$
 is given by
$$
\begin{array}{rl}
X^2 = & (A+\beta B)^2+\gamma^2 (C+\beta D)^2
+\gamma^4(E+\beta F)^2\\
=& \Big( (A_{>} + \ ^t\!B_{>} \times T) +\beta B_{>} \Big)+ \gamma
\Big(\  ^t\!F_{>}  \times T+  \beta ( E_{>}+ F_{>} +\  ^t\!F_{>}  \times T)\Big)\\
 + & \gamma^2 \Big( (C_{>} + \  ^t\!D_{>} \times T) + \beta D_{>} \Big),
\end{array}
$$
where $A_{>}, B_{>}, C_{>},D_{>},E_{>}$, $F_{>}$
stand for right-cyclic shifts of the coordinate vectors
of \ $A$, $B$, $C$, $D$, $E$, $F$, and 
 $$\  ^t\!X \times T$$
is a vector-matrix multiplication between the transpose of $X$
and the multiplication table of \ $\mathcal{N}$.
On the other hand, the product of two distinct elements 
$$
X_1=(A_1+\beta B_1)+\gamma (C_1+\beta D_1)
+\gamma^2(E_1+\beta F_1)   \ \text{ and } \
X_2=(A_2+\beta B_2)+\gamma (C_2+\beta D_2)
+\gamma^2(E_2+\beta F_2)
$$
is given by
$$
\small{\begin{array}{lll}
X_1 \times X_2 & = & (A_1+\beta B_1)(A_2+\beta B_2)\\
& +& \alpha \Bigg( (A_1+\beta B_1)(A_2+\beta B_2) 
 +  \Big((A_1+\beta B_1)+(C_1+\beta D_1)\Big)
\Big((A_2+\beta B_2)+(C_2+\beta D_2)\Big) \\
&+ &
\Big( (A_1+\beta B_1)+(E_1+\beta F_1)\Big)
\Big((A_2+\beta B_2)+(E_2+\beta F_2) \Big)\\
& + & 
\Big( (A_1+\beta B_1)+ (C_1+\beta D_1)+(E_1+\beta F_1)\Big)
\Big((A_2+\beta B_2)+(C_2+\beta D_2)+(E_2+\beta F_2) \Big) \Bigg)\\
&+& +\beta \Bigg( (A_1+\beta B_1)(A_2+\beta B_2)+ (C_1+\beta D_1)(C_2+\beta D_2)
\alpha (E_1+\beta F_1)(E_2+\beta F_2)\\
&+&  \Big((A_1+\beta B_1)+(C_1+\beta D_1)\Big)
\Big((A_2+\beta B_2)+(C_2+\beta D_2)\Big) \Bigg) \\
& + &  \beta^2 \Bigg( (A_1+\beta B_1)(A_2+\beta B_2) + (C_1+\beta D_1)(C_2+\beta D_2)
(E_1+\beta F_1)(E_2+\beta F_2)\\
& + & \Big( (A_1+\beta B_1)+(E_1+\beta F_1)\Big)
+\Big((A_2+\beta B_2)+(E_2+\beta F_2) \Big) \Bigg).
 \end{array}}
$$

Moreover, the density of $\mathcal{K}\!\mathcal{A}$ is computed from the following block matrix

\vspace{.25cm}

$$
\small{\left( \begin{array}{c|c|c|c|c|c}
& & & & & \\ 
(\alpha^{2^i}\alpha^{2^j}) & (\beta\alpha^{2^i}\alpha^{2^j})
& (\gamma\alpha^{2^i}\alpha^{2^j}) & (\gamma\beta\alpha^{2^i}\alpha^{2^j}) &
(\gamma^2\alpha^{2^i}\alpha^{2^j}) & (\gamma^2\beta\alpha^{2^i}\alpha^{2^j}) \\ 
& & & & &\\ \hline
& & & & & \\ 
(\beta\alpha^{2^i}\alpha^{2^j}) & (\beta^2\alpha^{2^i}\alpha^{2^j})
& (\gamma\beta\alpha^{2^i}\alpha^{2^j}) & (\gamma\beta^2\alpha^{2^i}\alpha^{2^j}) &
(\gamma^2\beta\alpha^{2^i}\alpha^{2^j}) & (\gamma^2\beta^2\alpha^{2^i}\alpha^{2^j}) \\ 
& & & & & \\ \hline
& & & & & \\ 
(\gamma\alpha^{2^i}\alpha^{2^j}) & (\gamma \beta\alpha^{2^i}\alpha^{2^j})
& (\gamma^2\alpha^{2^i}\alpha^{2^j}) & (\gamma^2\beta\alpha^{2^i}\alpha^{2^j}) &
(\beta\alpha^{2^i}\alpha^{2^j}) & (\beta^2\alpha^{2^i}\alpha^{2^j}) \\ 
& & & & &\\ \hline
& & & & & \\ 
(\gamma\beta\alpha^{2^i}\alpha^{2^j}) & (\gamma\beta^2\alpha^{2^i}\alpha^{2^j})
& (\gamma^2\beta\alpha^{2^i}\alpha^{2^j}) & (\gamma^2\beta^2\alpha^{2^i}\alpha^{2^j}) &
(\beta^2\alpha^{2^i}\alpha^{2^j}) & (\beta^3\alpha^{2^i}\alpha^{2^j}) \\ 
& & & & &\\ \hline
& & & & & \\ 
(\gamma^2\alpha^{2^i}\alpha^{2^j}) & (\gamma^2 \beta\alpha^{2^i}\alpha^{2^j})
& (\beta\alpha^{2^i}\alpha^{2^j}) & (\beta^2\alpha^{2^i}\alpha^{2^j}) &
(\gamma\beta\alpha^{2^i}\alpha^{2^j}) & (\gamma\beta^2\alpha^{2^i}\alpha^{2^j}) \\ 
& & & & &\\ \hline
& & & & & \\ 
(\gamma^2\beta\alpha^{2^i}\alpha^{2^j}) & (\gamma^2\beta^2\alpha^{2^i}\alpha^{2^j})
& (\beta^2\alpha^{2^i}\alpha^{2^j}) & (\beta^3\alpha\alpha^{2^i}\alpha^{2^j}) &
(\gamma\beta^2\alpha^{2^i}\alpha^{2^j}) & 
(\gamma\beta^3\alpha\alpha^{2^i}\alpha^{2^j}) \\ 
& & & & &\\
\end{array} \right)}
$$

\vspace{.5cm}

If the original normal basis $\mathcal{N}$ has quasi-linear 
complexity $O(n\log n \vert \log \log n\vert)$
and linear weight $O(n)$,
then $\mathcal{A}$ has also quasi-linear complexity and 
its density is quadratic by Proposition \ref{prop:1}.
From argument analogous to the one used in the proof
of Proposition \ref{prop:4},
we deduce that $\mathcal{K}\!\mathcal{A}$ has quasi-linear complexity.
Obviously, degree $d$ Kummer extensions of an Artin-Schreier extension of
 a normal basis $\mathcal{N}$
of $\bF_{\!2^n}/\bF_{\!2}$
are  useful when doing arithmetic in
$\bF_{\!2^{2nd}}$ provided that $d$ is not too large.
In order to fully take advantage of properties of the original
normal basis, we are only authorized
to construct Kummer extended bases with low degrees.

We know (from equation $(\ref{eq:17})$) that the complexity
of a multiplication algorithm using multiplication tables
of a basis $\mathcal{B}$ depends on the density of
$\mathcal{B}$. So density is an important criterion
when selecting efficient extended bases.
Since polynomials of the form
$X^3+\alpha$ are sparser than the ones of the form $X^2+X+\alpha$,
we guess that there are many cases for which extended bases constructed from Kummer
theory have better densities than the ones from Artin-Schreier theory.
We used Magma \cite{Magma}  to construct
Table \ref{Tableau1} which confirm our guess by comparing the densities of 
the best known Kummer extended bases to  the densities of the
best known Artin-Schreier extended bases
and  the densities of the best known normal bases of $\bF_{\!2^m}/\bF_{\!2}$
in case $6\le m\le 78$. 
We observe that (when they exist) Kummer extended 
bases have better densities than
both others, except in cases $m\in \{18, 24\}$ for which
the densities of Kummer extended bases lie between the
densities of both others.

\bigskip

\vspace{.5cm}

\subsection*{Acknowledgments}
The work reported in this paper is supported by Simons Foundation
via PREMA project, and the Inria International Lab LIRIMA via the Associate team FAST.
The first author acknowledges 
the International Centre for Theoretical Physics (ICTP) for their hospitality
within the framework of Associate Scheme.
The authors would like to thank Jean-Marc Couveignes for his comments on early version of this work.
We also thank the anonymous referee for
various comments that were helpful
for the improvement of the exposition.

\begin{table}
\caption{Best known densities $d(\mathcal{K})$ of Kummer extended bases,
versus best known densities $d(\mathcal{A})$ of Artin-Schreier extended bases
 and best known densities $d(\mathcal{N})$ of normal bases
of $\bF_{\!2^m}/\bF_{\!2}$, for integers $6\le m\le 78$ which are multiple of $6$.
All these bases are computed from the best known normal elements given in tables \ref{Tableau2}
and \ref{Tableau3}.
Bold entries indicate that Kummer extended bases are better than both others.
Minus symbol indicates that there is no normal element  in tables \ref{Tableau2}
or \ref{Tableau3} which generates a normal basis admiting a 
degree $3$ Kummer extension.
Blanc indicates that data are not available in the literature.}\label{Tableau1}
$\begin{array}{lll}
{\begin{array}{|c|c|c|c|}\hline
m & d(\mathcal{N})& d(\mathcal{A})& d(\mathcal{K})\\ \hline
6&66& 77 &\mathbf{51}\\ \hline
12&276& 365 &$-$ \\ \hline
18&630& 869 & 699 \\ \hline
24&2520& 1369 & 1707 \\ \hline
30&1770& 3805 & $-$ \\ \hline
36&2556&3133 & $-$ \\ \hline
42& 5670&10921 & \mathbf{4299} \\ \hline
48&20400&14041 & \mathbf{13923} \\ \hline
54&11286&23245 & $-$  \\ \hline
60& 7140& 10445 & $-$ \\ \hline
66& 8646&12677 &  $-$ \\ \hline
72&25704 & & $-$ \\ \hline
78& 18018& &\mathbf{15459} \\ \hline
\end{array}
}
\end{array}
$
\end{table}

 \vspace{2.58cm} 

\bibliographystyle{plain}

\bibliography{biblio_FF_Arithmetic_Char2}

\end{document}